\documentclass[12pt]{amsart}
\usepackage{amssymb,amsfonts,amsmath,amsthm,cite,verbatim}
\usepackage{xcolor}
\usepackage{tikz}
\usepackage{graphicx}
\usepackage[left=2.5cm, right=2.5cm, top=3.5cm, bottom=3.5cm]{geometry}

\usepackage{hyperref}

\usepackage[normalem]{ulem}

\numberwithin{equation}{section}
\newtheorem{theorem}{Theorem}[section]
\newtheorem{corollary}[theorem]{Corollary}

\newcommand{\abs}[1]{\lvert#1\rvert}

\newcommand{\qhyp}[5]{
{_{#1}\phi_{#2}}\bigg(\genfrac{}{}{0pt}{}{#3}{#4};#5\bigg)}

\begin{document}

\title{An extension of the Andrews--Warnaar partial theta function identity}

\author{Lisa H.\ Sun}

\address{Center for Combinatorics, LPMC-TJKLC,
Nankai University, Tianjin 300071, P.R. China}

\email{sunhui@nankai.edu.cn}


\begin{abstract}
In this paper, by  applying  a range of classic summation and 
transformation formulas for basic hypergeometric series,
we obtain a three-term identity for partial theta
functions. It extends the Andrews--Warnaar partial theta 
function identity, and also unifies several results on partial theta functions
due to Ramanujan, Lovejoy and Kim.
We also establish a two-term version of the extension, 
which can be used to derive identities for
partial and false theta functions.
Finally, we present
a relation between the big $q$-Jacobi 
polynomials and the Andrews--Warnaar partial theta function 
identity.

\end{abstract} 

\maketitle

\allowdisplaybreaks

\section{Introduction}\label{sec1}

Throughout this paper, we adopt  standard notation and terminology
for $q$-series \cite{GR04}. 
The $q$-shifted factorial is defined by
\[
(a;q)_n=\begin{cases}
1, & \text{if $n=0$}, \\[2mm]
(1-a)(1-aq)\cdots(1-aq^{n-1}), & \text{if $n\geq 1$}.
\end{cases}
\]
We also use the notation
\[
(a;q)_\infty=\prod_{n=0}^\infty (1-aq^n),
\]
where $\abs{q}<1$. 
There are more compact notations for the multiple $q$-shifted factorials:
\begin{align*}
(a_1,a_2,\dots,a_m;q)_n&=(a_1;q)_n(a_2;q)_n \cdots(a_m;q)_n,\\
(a_1,a_2,\dots,a_m;q)_{\infty}&=(a_1;q)_{\infty}(a_2;q)_{\infty}\cdots
(a_m;q)_{\infty}.
\end{align*}

Andrews \cite{And81} defined partial theta functions as sums of the form
\[
\sum_{n=0}^\infty q^{An^2+Bn}x^n,
\]
in which $A>0$ and the sum over $\mathbb{Z}$ defining an ordinary theta 
function is replaced by a sum over the ‘positive cone’ 
$\{n \in \mathbb{Z}\colon n\geq 0\}$.

In Ramanujan's Lost Notebook, there are a number of partial theta function
identities such as~\cite[p.~37]{Ram88}
\begin{multline}\label{ram37}
\sum_{n=0}^\infty \frac{q^n}{(aq,q/a;q)_n}=(1-a) 
\sum_{n=0}^\infty (-1)^na^{3n}q^{n(3n+1)/2}(1-a^2q^{2n+1})\\
+\frac{a}{(aq,q/a;q)_\infty} \sum_{n=0}^\infty (-1)^n a^{2n}q^{\binom{n+1}{2}}.
\end{multline} 
This and Ramanujan's other partial theta function 
identities were proved by Andrews~\cite{And81}. 
His main tools was the following general identity~\cite[Theorem 1]{And81} 
\begin{multline}\label{andcd}
\sum_{n=0}^\infty \frac{(c,d;q)_nq^n}{(aq,bq;q)_n}=
(1-a)\sum_{n=0}^\infty \frac{(1/b;q)_{n+1}(cd/ab;q)_n a^n}
{(c/b,d/b;q)_{n+1}} \\
 +\frac{(c,d;q)_\infty}{b(aq,bq;q)_\infty} 
\sum_{n=0}^\infty \frac{(aq/c;q)_n}{(d/b;q)_{n+1}}
\Big(\frac{c}{b}\Big)^n,
\end{multline}
where $\abs{a}<1$ and $\abs{c/b}<1$. 
For example, by taking the limit as $c,d \to 0$ in \eqref{andcd}, 
setting $b\mapsto 1/a$  and then transforming the first term on the right 
using the Rogers--Fine identity~\cite[Eq.~(14.1)]{Fine88}, it simplifies 
to~\eqref{ram37}. As pointed out by Warnaar \cite{War19}, the exception is~\cite[p.~12]{Ram88}
\begin{equation}\label{Ram88}
\sum_{n=0}^\infty \frac{(q^{n+1};q)_n q^n}{(aq,q/a;q)_n}
=(1-a)\sum_{n=0}^\infty a^n q^{n^2+n}+
\frac{a}{(aq,q/a;q)_\infty} \sum_{n=0}^\infty a^{3n}q^{n(3n+2)}(1-aq^{2n+1}),
\end{equation}
in that it is the only three-term partial theta function
identity from the Lost Notebook that does not follow from~\eqref{andcd}.
It follows as a simple consequence of the main result presented in this paper,
stated as Theorem~\ref{mainthm} below. 

To be compared with \eqref{Ram88}, Warnaar \cite[(4.13)]{War03} discovered that
\begin{equation}\label{war03}
\sum_{n=0}^\infty \frac{(q^{n+1};q)_n q^n}{(q,aq,q/a;q)_n}=(1-a)
\sum_{n=0}^\infty (-a)^n q^{\binom{n+1}{2}} + \frac{a}{(q,aq,q/a;q)_\infty} 
\bigg(\sum_{n=0}^\infty (-a)^n q^{\binom{n+1}{2}}\bigg)^2.
\end{equation} 
In the same paper \cite{War03}, 
he also gave an extension of Jacobi's triple product identity as follows
\begin{equation}\label{genjac}
1+\sum_{n=1}^\infty (-1)^n q^{\binom{n}{2}} (a^n+b^n)=
(q,a,b;q)_\infty\sum_{n=0}^\infty  \frac{(ab/q;q)_{2n}q^n}{(q,a,b,ab;q)_n}.
\end{equation}
Together with the Bailey lemma, 
the above identity can be used to prove each of Ramanujan's partial theta 
function identities as well as  embed each such identity into an infinite 
family.
Subsequently, Andrews and Warnaar \cite{AndWar07} proved an identity 
for the product of two partial theta functions
\begin{equation}\label{AndWar}
\bigg(\sum_{n=0}^\infty (-1)^n a^n q^{\binom{n}{2}}\bigg) 
\bigg(\sum_{n=0}^\infty (-1)^n b^n q^{\binom{n}{2}}\bigg)=
(q,a,b;q)_\infty\sum_{n=0}^\infty  \frac{(abq^{n-1};q)_nq^n}{(q,a,b;q)_n},
\end{equation}
and showed that \eqref{genjac} is a consequence of~\eqref{AndWar}. 
The two closely related identities \eqref{genjac} 
and \eqref{AndWar} motivated  many  variations and generalisations, 
see, for example,~\cite{Ber07,Ma12,SchWar02,WangMa18,Wei18}.

In \cite{And84}, Andrews observed that one can derive non-trivial 
$q$-series identities by calculating the residue around the pole $a=q^N$ in 
Ramanujan's partial theta function identities and by then invoking
analyticity to replace $q^N$ by $a$.
Based on Andrews and Warnaar's works, Kim and Lovejoy \cite{Lov12, KL18} 
also obtained many residual identities and 
extracted new conjugate Bailey pairs from
them.

By applying a range of summation and transformation formulas 
for basic hypergeometric series, we obtain 
the following extension of identity~\eqref{AndWar} due to 
Andrews and Warnaar.
\begin{theorem}\label{mainthm} 
We have
\begin{align}
\sum_{n=0}^\infty & 
\frac{(c,d;q)_n(abq^{n-1};q)_{n}q^n}{(q,a,b;q)_n}\label{gfcd}\\
&=\frac{(q/a,cdq/a;q)_\infty}{(cq/a,dq/a;q)_\infty}
\sum_{n=0}^\infty \frac{(c,d,cq/a,dq/a;q)_nb^nq^{n^2}}
{(q,b;q)_n(cdq/a;q)_{2n}}\nonumber\\[5pt] &\quad
+\frac{q}{a}\,\frac{(c,d;q)_\infty}{(q,a;q)_\infty} 
\sum_{n=0}^\infty \frac{b^nq^{n^2}}{(q,b;q)_n}
\sum_{k=0}^\infty \frac{(q^{1-n}/c,q^{1-n}/d;q)_kq^{k^2+2nk+2k}
(1-q^{2k+2}/a)}{(cq^{n+1}/a,dq^{n+1}/a;q)_{k+1}} 
\Big(\frac{cd}{a^2}\Big)^k.  \nonumber
\end{align}
\end{theorem}

Identity~\eqref{AndWar} follows from Theorem~\ref{mainthm}
by taking the limit $c,d\to 0$  followed by some simple manipulations,
as will be shown  in Section~\ref{section2}.
We also note that by letting 
$c$ tend to zero in \eqref{gfcd}, then substituting  
$(a,b,d)\mapsto (q/a, aq, q)$, interchanging the 
order of the sums in the second term on the right hand side, 
and finally simplifying the resulting sum using $q$-Gauss sum
\cite[(II.8)]{GR04}, we recover~\eqref{Ram88}. 
Moreover, by taking the limits as $c,d \to 0$ and then replacing 
$(a,b)\mapsto (q/a,aq)$ in \eqref{gfcd}, we obtain the identity~\eqref{war03}. 

In Section~\ref{section2}, we give a proof of Theorem~\ref{mainthm}
and state additional partial theta function identities as special cases. 
In Section~\ref{section3}, we observe that there is a two-term version of 
identity \eqref{gfcd}, from which we can derive Ramanujan-type identities 
for partial and false theta functions. This two-term version also recovers 
residual identities given by Warnaar and Lovejoy. In Section~\ref{section4},
we describe a relation between the big $q$-Jacobi polynomials and
the Andrews--Warnaar partial theta function identity~\eqref{AndWar}.

\section{Proof of Theorem~\ref{mainthm}}\label{section2}

In this section, we give a detailed proof of Theorem~\ref{mainthm},
and apply the theorem to derive partial theta function identities. 

Recall that the
basic hypergeometric series $_r\phi_s$ is defined as follows:
\[
\qhyp{r}{s}{a_1,a_2,\dots,a_r}{b_1,b_2,\dots,b_s}{q,x}
=\sum_{n=0}^{\infty}\frac{(a_1,a_2,\dots,a_r;q)_n}
{(q,b_1,\dots,b_s;q)_n} \Big[(-1)^n q^{\binom{n}{2}}\Big]^{1+s-r} x^n.
\]
To prove the results in this paper, we will need the following summation and transformation 
formulas for basic hypergeometric series. The $q$-binomial theorem 
is~\cite[(II.4)]{GR04}
\begin{equation}\label{qbinom}
\qhyp{1}{0}{a}{-}{q,z}=\frac{(az;q)_\infty}{(z;q)_\infty},
\end{equation}
where $\abs{z}<1$.
The $q$-Chu--Vandermonde sum is~\cite[(II.6)]{GR04}
\begin{equation}\label{qvan}
\qhyp{2}{1}{a,q^{-n}}{c}{q,q}=\frac{(c/a;q)_n}{(c;q)_n}\,a^n.
\end{equation}
The $q$-Gauss summation for $_2\phi_1$ series is~\cite[(II.8)]{GR04}
\[
\qhyp{2}{1}{a,b}{c}{q,\frac{c}{ab}}
=\frac{(c/a,c/b;q)_\infty}{(c,c/ab;q)_\infty},
\]
where $\abs{c/ab}<1$.
Note that in the limit $a,b\to \infty$, 
the $q$-Gauss sum simplifies to
\begin{equation}\label{limit2phi1}
\sum_{n=0}^\infty \frac{q^{n^2-n}c^n}{(q,c;q)_n}=\frac{1}{(c;q)_\infty}.
\end{equation}
Three well known transformations for $_2\phi_1$ series due to
Heine are~\cite[(III.1)--(III.3)]{GR04}
\begin{subequations}
\begin{align}
\qhyp{2}{1}{a,b}{c}{q,z}&=\frac{(b,az;q)_\infty}{(c,z;q)_\infty}\,
\qhyp{2}{1}{c/b,z}{az}{q,b}
\label{heine1} \\[5pt]
&=\frac{(c/b,bz;q)_\infty}{(c,z;q)_\infty}\,  
\qhyp{2}{1}{abz/c,b}{bz}{q,\frac{c}{b}}
\label{heine2}\\[5pt]
&=\frac{(abz/c;q)_\infty}{(z;q)_\infty}\,
\qhyp{2}{1}{c/a,c/b}{c}{q,\frac{abz}{c}}
\label{heine3}
\end{align}
\end{subequations}
provided that all ${_2\phi_1}$ series converge.
By substituting $(c,z)\mapsto (cq,c/ab)$ in \eqref{heine2} and taking 
the limit as $a,b\to \infty$, we obtain
\begin{equation}\label{limitheine2}
\sum_{n=0}^\infty \frac{q^{n^2-n}c^n}{(q,cq;q)_n}=\frac{1+c}{(cq;q)_\infty}.
\end{equation}
The Rogers--Fine identity is \cite[(14.1)]{Fine88}
\begin{equation}\label{rogersf}
\sum_{k=0}^{\infty} \frac{(a;q)_k}{(b;q)_k}\, t^k=
\sum_{k=0}^{\infty} \frac{(a,atq/b ; q)_k  (bt)^k
q^{k^2-k} (1-atq^{2k})}{(b;q)_k(t;q)_{k+1}},
\end{equation}
where $\abs{t}<1$. 
Jackson's transformation formula for ${_2\phi_1}$ series is \cite[(III.4)]{GR04}
\begin{equation} \label{jacksontr}
\qhyp{2}{1}{a,b}{c}{q,z}=
\frac{(az;q)_\infty}{(z;q)_\infty}\,
\qhyp{2}{2}{a,c/b}{c,az}{q,bz},
\end{equation}
 where $\abs{z}<1$. 
One of the transformation formulas for ${_3\phi_2}$ series is \cite[(III.10)]{GR04}
\begin{equation}\label{3phi2tr}
\qhyp{3}{2}{a,b,c}{d,e}{q,\frac{de}{abc}}=
\frac{(b,de/ab,de/bc;q)_\infty}{(d,e,de/abc;q)_\infty}\,
\qhyp{3}{2}{d/b,e/b,de/abc}{de/ab,de/bc}{q,b} 
\end{equation}
provided that $\max\{\abs{b},\abs{de/abc}\}<1$.
A three-term transformation 
 formula for $_2\phi_1$ series 
 is \cite[(III.31)]{GR04}
\begin{multline} \label{3term2phi1}
\qhyp{2}{1}{a,b}{c}{q,z}
=\frac{(abz/c,q/c;q)_{\infty}}{(az/c,q/a;q)_{\infty}}\, 
\qhyp{2}{1}{c/a,cq/abz}{cq/az}{q,\frac{bq}{c}}\\[5pt]
-\frac{(b,q/c,c/a,az/q,q^2/az;q)_{\infty}}
{(c/q,bq/c,q/a,az/c,cq/az;q)_{\infty}} \, 
\qhyp{2}{1}{aq/c,bq/c}{q^2/c}{q,z},
\end{multline} 
where $\max\{\abs{z}, \abs{bq/c}\}<1$.

Now we are ready to give the proof of Theorem~\ref{mainthm}. 

\begin{proof}[Proof of Theorem~\ref{mainthm}]
By the $q$-Chu--Vandermonde sum~\eqref{qvan} with 
$(a,c) \mapsto (q^{1-n}/b,a)$, we  find that
\[
\qhyp{2}{1}{q^{1-n}/b,q^{-n}}{a}{q,q}
=\frac{(abq^{n-1};q)_n}{(a;q)_n} \, (q^{1-n}/b)^n.
\]
Hence it immediately follows that
\[
\sum_{n=0}^\infty \frac{(c,d,abq^{n-1};q)_n z^n}
{(q,a,b;q)_n} = 
\sum_{n=0}^\infty \frac{(c,d;q)_n }{(q,b;q)_n}\, 
q^{n^2-n} (bz)^n
\sum_{k=0}^n \frac{(q^{1-n}/b,q^{-n};q)_k}{(q,a;q)_k}\,q^k.
\]
This can be rewritten by interchanging the order of the sums on 
the right-hand side and then shifting the summation index 
$n\mapsto n+k$. Thus
\begin{align}
\label{rhsthm}
\sum_{n=0}^\infty & \frac{(c,d,abq^{n-1};q)_n z^n }{(q,a,b;q)_n}
\\[5pt]
& = \sum_{k=0}^\infty \sum_{n=0}^\infty 
\frac{(c,d;q)_{n+k} }{(q;q)_{n}(b;q)_{n+k}} \,
q^{(n+k)^2-n} (bz)^{n+k}
 \frac{(-1)^kq^{-(n+k)k+{\binom{k}{2}}}(q^{-n-k+1}/b;q)_k}
{(q,a;q)_k}\nonumber\\[5pt]
 &=\sum_{n=0}^\infty \frac{(c,d;q)_n}{(q,b;q)_n} \, q^{n^2-n} (bz)^n
  \sum_{k=0}^\infty \frac{(cq^n,dq^n;q)_k z^k}{(q,a;q)_k}\nonumber\\[5pt]
&=\sum_{n=0}^\infty  \frac{(c,d;q)_n}{(q,b;q)_n} \, q^{n^2-n} (bz)^n \,
\qhyp{2}{1}{cq^n,dq^n}{a}{q,z}. 
\notag
\end{align}

By applying the three-term transformation 
formula for $_2\phi_1$ series \eqref{3term2phi1}
with $(a,b,c) \mapsto (cq^n,dq^n,a)$, we obtain
\begin{multline*}
\sum_{n=0}^\infty  \frac{(c,d;q)_n(abq^{n-1};q)_{n}z^n}{(q,a,b;q)_n} \\
=\frac{(q/a,cdz/a;q)_\infty}{(cz/a,q/c;q)_\infty}\sum_{n=0}^\infty 
\frac{(c,d,cz/a;q)_nq^{n^2-n}(bz)^n}{(q,b,q^{1-n}/c;q)_n(cdz/a;q)_{2n}}\, 
\qhyp{2}{1}{aq^{-n}/c,aq^{1-2n}/cdz}{aq^{1-n}/cz}{q,\frac{dq^{n+1}}{a}} \\[5pt]
+\frac{q}{a}\,\frac{(d,q^2/a,a/c,cz/q,q^2/cz;q)_\infty}
{(a,dq/a,q/c,cz/a,aq/cz;q)_\infty} 
\sum_{n=0}^\infty \frac{q^{n^2-n}(bz)^n(c,aq^{-n}/c,q^{2-n}/cz,dq/a,cz/a;q)_n}
{(q,b,cz/q,q^{1-n}/c,aq^{1-n}/cz;q)_n}\\
 \times  \qhyp{2}{1}{cq^{n+1}/a,dq^{n+1}/a}{q^2/a}{q,z}.
\end{multline*}
Denote the two terms on the right hand side of the above identity by $R_1$ 
and $R_2$, respectively. 
Now, by setting $z=q$ in $R_1$ and by using the 
$q$-binomial theorem~\eqref{qbinom}, $R_1$ reduces to 
\begin{align*}
&\frac{(q/a,cdq/a;q)_\infty}{(cq/a,q/c;q)_\infty}\sum_{n=0}^\infty 
\frac{(c,d,cq/a;q)_nb^nq^{n^2}}{(q,b,q^{1-n}/c;q)_n(cdq/a;q)_{2n}}\,
\qhyp{1}{0}{aq^{-2n}/cd}{-}{q,\frac{dq^{n+1}}{a}} \\
&\qquad=\frac{(q/a,cdq/a;q)_\infty}{(cq/a,q/c;q)_\infty}
\sum_{n=0}^\infty \frac{(c,d,cq/a;q)_nb^nq^{n^2}}
{(q,b,q^{1-n}/c;q)_n(cdq/a;q)_{2n}}\,
\frac{(q^{1-n}/c;q)_\infty}{(dq^{n+1}/a;q)_\infty}\\
&\qquad=\frac{(q/a,cdq/a;q)_\infty}{(cq/a,dq/a;q)_\infty}
\sum_{n=0}^\infty \frac{(c,d,cq/a,dq/a;q)_nb^nq^{n^2}}
{(q,b;q)_n(cdq/a;q)_{2n}},
\end{align*}
which is the first term on the right hand side of \eqref{gfcd}. 

When we specialise $z=q$ in $R_2$, we are led to 
\begin{equation}\label{R2zq}
\frac{q}{a}\,\frac{(c,d,q^2/a;q)_\infty}{(a,cq/a,dq/a;q)_\infty} 
\sum_{n=0}^\infty \frac{(cq/a,dq/a;q)_nb^nq^{n^2}}{(q,b;q)_n}\,
\qhyp{2}{1}{cq^{n+1}/a,dq^{n+1}/a}{q^2/a}{q,q}.
\end{equation}
Further applying Heine's transformation formula~\eqref{heine1} 
with $(a,b,c,z)$ $\mapsto$ $(cq^{n+1}/a$, $dq^{n+1}/a,$ $q^2/a,q)$ to the 
$_2\phi_1$ series, it follows that
\begin{align*}
\qhyp{2}{1}{cq^{n+1}/a,dq^{n+1}/a}{q^2/a}{q,q}
&=\frac{(cq^{n+2}/a,dq^{n+1}/a;q)_\infty}{(q^2/a,q;q)_\infty} \, 
\qhyp{2}{1}{q^{1-n}/d,q}{cq^{n+2}/a}{q,\frac{dq^{n+1}}{a}} \\[5pt]
&=\frac{(dq^{n+1}/a,cq^{n+2}/a;q)_\infty}{(q^2/a,q;q)_\infty}
\sum_{k=0}^\infty \frac{(q^{1-n}/d;q)_k}{(cq^{n+2}/a;q)_k}
\bigg(\frac{dq^{n+1}}{a}\bigg)^k.
\end{align*}
Then employing the Rogers--Fine identity \eqref{rogersf}
with $(a,b,t)\mapsto(q^{1-n}/d,cq^{n+2}/a,dq^{n+1}/a)$, we finally obtain
\begin{multline*}
\qhyp{2}{1}{cq^{n+1}/a, dq^{n+1}/a}{q^2/a}{q,q}
=\frac{(dq^{n+1}/a,cq^{n+1}/a;q)_\infty}{(q^2/a,q;q)_\infty}\\
\times \sum_{k=0}^\infty \frac{(q^{1-n}/c,q^{1-n}/d;q)_k
q^{k^2+2nk+2k}(1-q^{2k+2}/a)}{(cq^{n+1}/a,dq^{n+1}/a;q)_{k+1}} 
\bigg(\frac{cd}{a^2}\bigg)^k.
\end{multline*}
Substituting the above expression into \eqref{R2zq}, we see that $R_2$ 
turns to be the second term on the right hand side of \eqref{gfcd},
which completes the proof.
\end{proof}

Now we are ready to show how to obtain the
Andrews--Warnaar partial theta function identity \eqref{AndWar} from 
Theorem \ref{mainthm}.

\begin{proof}[Proof of~\eqref{AndWar}] 
When $c,d \to 0$ in \eqref{gfcd},  we obtain
\begin{multline*}
\sum_{n=0}^{\infty} \frac{(abq^{n-1};q)_n q^n}{(q,a,b;q)_n} 
=(q/a;q)_{\infty}\sum_{n=0}^\infty \frac{b^nq^{n^2}}{(q,b;q)_n}\\
+\frac{1}{(q,a;q)_{\infty}} 
\sum_{n=0}^{\infty} \frac{b^nq^{n^2}}{(q,b;q)_n}
\sum_{k=0}^{\infty} a^{-2k-1}q^{2k^2+3k+1}(1-q^{2k+2}/a). 
\end{multline*}
By setting $z\mapsto q$ and taking the limit $a,b \to 0$ in Heine's 
transformation formulas \eqref{heine1} and \eqref{heine3}, it follows that
\begin{equation}\label{specialheine}
\sum_{n=0}^\infty \frac{c^nq^{n^2}}{(q,c;q)_n}=
\frac{1}{(c;q)_\infty}\sum_{k=0}^\infty (-1)^k q^{\binom{k}{2}} c^k.
\end{equation}
By noting that 
\[
\sum_{k=0}^\infty a^{-2k-1}q^{2k^2+3k+1}(1-q^{2k+2}/a)=
-\sum_{k=-\infty}^{-1} (-a)^k q^{\binom{k}{2}},
\]
and applying \eqref{specialheine}, we have 
\begin{align*}
\sum_{n=0}^{\infty}  \frac{(abq^{n-1};q)_n q^n}{(q,a,b;q)_n} 
&=\frac{(q/a;q)_{\infty}}{(b;q)_\infty}
\sum_{n=0}^{\infty} (-b)^n q^{\binom{n}{2}}
-\frac{1}{(q,a,b;q)_{\infty}} 
\sum_{n=0}^\infty (-b)^n q^{\binom{n}{2}}
\sum_{k=-\infty}^{-1} (-a)^k q^{\binom{k}{2}}\\
&=\frac{1}{(q,a,b;q)_\infty} 
\Big((q,a,q/a;q)_\infty-
\sum_{k=-\infty}^{-1} (-a)^k q^{\binom{k}{2}}\Big) 
\sum_{n=0}^{\infty} (-b)^n q^{\binom{n}{2}}\\
&=\frac{1}{(q,a,b;q)_\infty}
\sum_{n=0}^{\infty} (-a)^n q^{\binom{n}{2}} 
\sum_{n=0}^{\infty} (-b)^n q^{\binom{n}{2}}.
\end{align*}
Here the last equality  follows from the Jacobi triple product identity,
\[
\sum_{k=-\infty}^\infty (-1)^k q^{\binom{k}{2}} a^k = (q,a,q/a;q)_\infty.
\]
This completes the proof of~\eqref{AndWar}.
\end{proof}

We remark that it is more direct to 
derive \eqref{AndWar} by letting $c,d \to 0$ in \eqref{rhsthm} 
and then simplifying by \eqref{specialheine}.

As a second application of Theorem~\ref{mainthm}, 
by setting $(a,b,c,d)\mapsto (q/a,aq,q,q)$ in \eqref{gfcd}, 
we obtain the following result,  which was apparently missed by Ramanujan.

\begin{corollary} We have
\[
\sum_{n=0}^\infty  \frac{(q;q)_{2n}q^n}{(aq,q/a;q)_n}=
(1-a)\sum_{n=0}^\infty \frac{(q;q)_nq^{n^2+n}a^n}{(aq^{n+1};q)_{n+1}}+
\frac{a(q;q)_\infty}{(aq,q/a;q)_\infty}
\sum_{n=0}^\infty (-1)^na^{2n} q^{3n(n+1)/2}.
\]
\end{corollary}

 From the work of Kim and Lovejoy 
\cite{KL17}, it follows that the first 
term on the right hand side has a representation as an indefinite
partial theta series as follows
\[
\sum_{n=0}^\infty \frac{(q;q)_nq^{n^2+n}a^n}{(aq^{n+1};q)_{n+1}}=
\sum_{\substack{r,s\geq 0\\r\equiv s\pmod{2}}} 
(-1)^ra^{(r+s)/2}q^{3rs/2+r/2+s}.
\]

By letting $c$ tend to zero and setting $(a,b,d)\mapsto (q/b,aq,aq/b)$ in \eqref{gfcd},
we obtain the following generalization of Ramanjuan's 
partial theta function identity~\eqref{Ram88}.

\begin{corollary} 
We have
\begin{multline*}
\sum_{n=0}^\infty  \frac{(aq/b;q)_{2n}q^n}{(q,aq,q/b;q)_n}
=\frac{(b;q)_\infty}{(aq;q)_\infty}
\sum_{n=0}^\infty \frac{(aq/b;q)_na^nq^{n^2+n}}{(q;q)_n}\\
+\frac{(aq/b;q)_\infty}{(q,q/b;q)_\infty} 
\sum_{n=0}^\infty \sum_{k=0}^\infty \frac{(-1)^k(bq^{-n}/a;q)_k
a^{n+k} b^{k+1} q^{n^2+n+nk+\frac{3k(k+1)}{2}}(1-bq^{2k+1})}
{(q;q)_n(aq;q)_{n+k+1}}.
\end{multline*}
\end{corollary}
When $b=a$ in the above identity,  the last term 
simplifies to
\begin{multline*}
\frac{1}{(q/a;q)_\infty} 
\sum_{n=0}^\infty \sum_{k=0}^n \frac{(-1)^k(q^{-n};q)_k
a^{n+2k+1} q^{n^2+n+nk+3k(k+1)/2}(1-aq^{2k+1})}
{(q;q)_n(aq;q)_{n+k+1}}\\[5pt]
= \frac{1}{(q/a;q)_\infty} \sum_{k=0}^\infty \frac{ q^{k(3k+2)}(1-aq^{2k+1})a^{3k+1}}{(aq;q)_{2k+1}} \sum_{n=0}^\infty 
\frac{q^n (aq^{2k+1})^n}{(q,aq^{2k+2};q)_n},
\end{multline*} 
which leads to \eqref{Ram88} by applying the 
limiting case of $q$-Gauss 
summation \eqref{limit2phi1} to the above  sum over $n$.

As another direct specialization of \eqref{gfcd}, 
by letting $c,d \to 0$ and substituting 
$(a,b)\mapsto (q^2/a, aq)$, we recover to the following identity given by 
Lovejoy \cite[(2.35)]{Lov12},
\begin{multline*}
\sum_{n=0}^\infty \frac{(q^{n+2};q)_{n}q^{n+1}}
{(q,aq;q)_n(q/a;q)_{n+1}}+(1-a)
\sum_{n=1}^{\infty} (-1)^n 
a^{n}q^{\binom{n}{2}}\\
=\frac{a}{(q,aq,q/a;q)_\infty}
\sum_{n_1,n_2\geq 0} (-a)^{n_1+n_2}q^{{\binom{n_1}{2}}+{\binom{n_2+1}{2}}}.
\end{multline*}

With a bit more work we can also obtain a second identity of
Lovejoy \cite[(2.33)]{Lov12}.

\begin{corollary}
We have
\begin{multline}\label{lov33}
\sum_{n=0}^\infty  \frac{(q^{n+1};q)_{n+1} q^n}{(aq;q)_n(q/a;q)_{n+1}}+
(1-a)\sum_{n=0}^\infty a^{n+1}q^{n^2+n-1}(1-q^{n+1}) \\ 
=\frac{1}{(aq,q/a;q)_\infty}
\sum_{n=0}^\infty a^{3n} q^{3n^2-2n-1}(1-aq^{2n})(1-q^n+aq^{2n}).
\end{multline}
\end{corollary}

\begin{proof} 
Observing that
\begin{equation}\label{left33}
\sum_{n=0}^\infty \frac{(q^{n+1};q)_{n+1}q^n}{(aq;q)_n(q/a;q)_{n+1}}=
\frac{(1-q)}{(1-q/a)}\sum_{n=0}^\infty 
\frac{(q^2;q)_n (q^{n+2};q)_n q^n}{(q;q)_n (aq,q^2/a;q)_n},
\end{equation}
we  take the limit $d\to 0$ 
in \eqref{gfcd} and substitute  $(a,b,c)\mapsto (q^2/a,aq,q^2)$.
 If we denote the resulting identity
by $L=R_1+R_2$, we see that
\[
L=\sum_{n=0}^\infty \frac{(q^2;q)_n (q^{n+2};q)_n q^n}{(q;q)_n (aq,q^2/a;q)_n}.
\]
The first term on the right is
\begin{align}\label{1term33}
R_1&=\frac{(a/q;q)_\infty}{(aq;q)_\infty} \sum_{n=0}^\infty \frac{a^nq^{n^2+n}(1-q^{n+1})}{(1-q)}\nonumber\\[5pt]
&=(1-a)\, \frac{1-a/q}{1-q} \sum_{n=0}^\infty a^nq^{n^2+n}(1-q^{n+1}).
\end{align}
Moreover, we can manipulate $R_2$ as
\begin{align*}
R_2&=\frac{a}{q} \, \frac{1}{(1-q)(q^2/a;q)_\infty} 
\sum_{n=0}^\infty \frac{a^nq^{n^2+n}}{(q,aq;q)_n} \sum_{k=0}^{n+1} \frac{(-1)^k (q^{-n-1};q)_k q^{k^2+nk+k+\binom{k}{2}}a^{2k}(1-aq^{2k})}{(aq^{n+1};q)_{k+1}}\\[5pt]
& = \frac{a}{q}\, \frac{1}{(1-q)(q^2/a;q)_\infty}\sum_{k=0}^\infty 
\sum_{n=0}^\infty \frac{a^{n+k-1}q^{3k^2-2k+n^2+2nk-n}a^{2k}(1-aq^{2k})(q;q)_{n+k}}{(q;q)_n(q;q)_{n+k-1}(aq;q)_{n+2k}}\\
& = \frac{a}{q}\, \frac{1}{(1-q)(q^2/a;q)_\infty}
\sum_{k=0}^\infty \frac{a^{3k-1}q^{3k^2-2k}(1-aq^{2k})}{(aq;q)_{2k}} 
\sum_{n=0}^\infty \frac{a^nq^{n^2-n+2nk}(1-q^{n+k})}{(q;q)_n(aq^{2k+1};q)_n}.
\end{align*}
 By substituting $c\mapsto cq$ in \eqref{limit2phi1} and combining with \eqref{limitheine2}, it follows that
\[
\sum_{n=0}^\infty \frac{c^n q^{n^2-n}(1-bq^n)}{(q,cq;q)_n}=\frac{1-b+c}{(cq;q)_\infty}.
\]
Letting $(b,c)\mapsto (q^k, aq^{2k})$ in the above identity, we see that 
\[
\sum_{n=0}^\infty \frac{a^nq^{n^2-n+2nk}(1-q^{n+k})}{(q;q)_n(aq^{2k+1};q)_n}=\frac{1-q^k+aq^{2k}}{(aq^{2k+1};q)_\infty}.
\]
Therefore,
\begin{align}\label{2term33}
R_2
& = \frac{1}{(1-q)(aq,q^2/a;q)_\infty}\sum_{k=0}^\infty a^{3k}q^{3k^2-2k-1}(1-aq^{2k})(1-q^k+aq^{2k}).
\end{align}
The proof of \eqref{lov33} is complete by combining \eqref{left33}, \eqref{1term33} and \eqref{2term33}.
\end{proof}

The last specialization of \eqref{gfcd} arises when $b,d\to 0$ and 
$(a,c)\mapsto (-q,q)$,
\[
\sum_{n=0}^\infty  \frac{q^n}{(-q;q)_n}=2-\frac{1}{(-q;q)_\infty}.
\]
This identity was given by  Warnaar \cite[p.~378]{War03} as a special case of the partial theta function identity  \cite[(4.15)]{War03}.

\section{A two-term version of Theorem~\ref{mainthm}}\label{section3}

In this section, we derive a two-term version of the extended 
Andrews--Warnaar partial theta function identity~\eqref{gfcd}.
It can be used to derive some of
Ramanjuan's identities on partial and 
false theta functions.
It also reduces to  residual identities due to 
Warnaar and Lovejoy.

From \eqref{rhsthm}, by employing 
Heine's transformation \eqref{heine1} 
with $(a,b,c,z)\mapsto (cq^n,dq^n,a,q)$, and then applying 
the Rogers--Fine identity \eqref{rogersf} with 
$(a,b,t)\mapsto(aq^{-n}/d, cq^{n+1}, dq^n)$, 
we obtain the following simplification
of \eqref{gfcd}.

\begin{theorem} We have
\begin{multline}
\sum_{n=0}^\infty \frac{(c,d;q)_n(abq^{n-1};q)_n q^n}{(q,a,b;q)_n}
\label{gfcd2}\\
=\frac{(c,d;q)_\infty}{(q,a;q)_\infty} 
\sum_{n=0}^\infty \frac{b^nq^{n^2}}{(q,b;q)_n}
\sum_{k=0}^\infty \frac{(aq^{-n}/d,aq^{-n}/c;q)_k (cd)^k 
q^{k^2+2nk}(1-aq^{2k})}{(cq^{n},dq^{n};q)_{k+1}}.  
\end{multline}
\end{theorem}

As a first application, 
we derive a two-term partial theta 
function identity from~\eqref{gfcd2}. 
We  begin with  an example from Ramanujan's 
Lost Notebook \cite[p.28]{Ram88}, 
see also Entry~{1.6.2} in~\cite{AndBer09}. 

\begin{corollary}
We have 
\begin{equation}\label{Ram38}
\sum_{n=1}^\infty \frac{(-a)^nq^{n(n+1)/2}(-q;q)_{n-1}}{(aq^2;q^2)_n}
=\sum_{k=1}^\infty (-a)^k q^{k^2}.
\end{equation}
\end{corollary}

We remark that Alladi \cite{All09} gave a partition theoretic interpretation
for this identity. 
For combinatorial proofs, see Alladi \cite{All10}, Berndt, Kim and 
Yee \cite{BKY10}, and Yee~\cite{Yee10}.

\begin{proof}
Denote the left hand side of \eqref{Ram38} by $L$.
Then  
\begin{align*}
L+\frac{1}{2}&=
\frac{1}{2}\sum_{n=0}^\infty  
\frac{(-a)^nq^{n(n+1)/2}(-1,q;q)_{n}}{(q,-q\sqrt{a}, q\sqrt{a};q)_n}\\[5pt]
&=\frac{1}{2}\,
\qhyp{2}{2}{-1,q}{-q\sqrt{a},q\sqrt{a}}{q,aq}.
\end{align*}
By employing Jackson's transformation formula~\eqref{jacksontr}
with $(a,b,c,z)\mapsto (-1,\sqrt{a},q\sqrt{a},q\sqrt{a})$, 
it follows that
\[
2L+1=\frac{(q\sqrt{a};q)_\infty}{(-q\sqrt{a};q)_\infty}\,
\qhyp{2}{1}{-1,\sqrt{a}}{q\sqrt{a}}{q,q\sqrt{a}}.
\]
Then applying Heine's transformation formula
\eqref{heine2} with $(a,b,c,z)\mapsto (-1,\sqrt{a},q\sqrt{a},q\sqrt{a})$,
this becomes
\[
2L+1=\frac{(q,aq;q)_\infty}{(aq^2;q^2)_\infty}\,
\qhyp{2}{1}{-\sqrt{a},\sqrt{a}}{aq}{q,q}.
\]
Thus by substituting $(a,b,c,d)\mapsto (aq,0,-\sqrt{a},\sqrt{a})$ 
in \eqref{gfcd2}, 
we obtain
\begin{align*}
2L+1&=\frac{(q,aq;q)_\infty}{(aq^2;q^2)_\infty} 
\frac{(-\sqrt{a},\sqrt{a};q)_\infty}{(q,aq;q)_\infty} 
\sum_{k=0}^\infty \frac{(-q\sqrt{a},q\sqrt{a};q)_k(-a)^kq^{k^2}
(1-aq^{2k+1})}{(-\sqrt{a},\sqrt{a};q)_{k+1}} \\[5pt]
&=\sum_{k=0}^\infty  (-a)^kq^{k^2}(1-aq^{2k+1})\\
&=1+2\sum_{k=1}^\infty (-a)^k q^{k^2},
\end{align*}
which completes the proof.
\end{proof}

False theta functions were introduced by Rogers in 1917 \cite{Rog17}
as series that appear like series for classical theta functions
except for incorrect signs of some of the terms in the series.
In his notebooks \cite{Ram57} 
as well as in the Lost Notebook \cite{Ram88}, 
Ramanujan gave many examples of identities for 
false theta functions. One of these identities was given as  Corollary (i)  of \cite[Entry 9]{Ber91}
which we will prove below as a consequence of \eqref{gfcd2}. See also, \cite[(2.8)]{Haj16} and more generally,
\cite[(7.1)]{BM15}. 

\begin{corollary} There holds
\begin{equation}\label{sgn}
\Big(\sum_{n\geq 0}-\sum_{n<0}\Big) q^{2n^2+n}
=(q;q)_\infty\sum_{n=0}^\infty \frac{q^{n^2+n}}{(q;q)_n^2}.
\end{equation}
\end{corollary}

\begin{proof}
By using Heine's transformation formula~\eqref{heine3} 
with $(a,b,c,z)\mapsto (q/a,q/b,q,ab)$ and then taking the limit as 
$a,b\to 0$, it follows that
\begin{equation}\label{sgn1}
\sum_{n=0}^\infty \frac{q^{n^2+n}}{(q;q)_n^2}=(q;q)_\infty \sum_{n=0}^\infty \frac{q^n}{(q;q)_n^2}.
\end{equation}
By taking the limit  as  $b,c,d \to 0$ and then setting $a\mapsto q$ in \eqref{gfcd2}, we get
\[
\sum_{n=0}^\infty \frac{q^n}{(q;q)_n^2}=\frac{1}{(q;q)_\infty^2} \sum_{n=0}^\infty q^{2n^2+n}(1-q^{2n+1}).
\]
Then \eqref{sgn} follows from substituting the above identity into 
\eqref{sgn1}.
\end{proof}

Note that the left hand side of \eqref{sgn} can be rewritten as
\[
\sum_{n=0}^\infty (-1)^n q^{\binom{n+1}{2}}.
\]
For this false theta  series, Ramanujan \cite{Ram88} gave the following identity on page 13 of his Lost Notebook, 
where he stated four more identities on false theta functions.

\begin{corollary}
There holds
\begin{equation}\label{andrews61}
\sum_{n=0}^\infty \frac{(q;q^2)_n(-1)^n q^{n^2+n}}{(-q;q)_{2n+1}}=\sum_{n=0}^\infty (-1)^nq^{\binom{n+1}{2}}.
\end{equation}
\end{corollary}

For other proofs of this identity, see also Andrews \cite[(6.1)]{And81}, 
Andrews and Warnaar \cite[(1.1a)]{AndWar072}, 
Andrews and Berndt \cite[Entry 9.3.2]{AndBer05} 
and Wang \cite[(1.1)]{Wang18}.

\begin{proof}
The left hand side of \eqref{andrews61} can be rewritten as follows 
\begin{align*}
\sum_{n=0}^\infty \frac{(q;q^2)_n(-1)^n q^{n^2+n}}{(-q;q)_{2n+1}}
&=\frac{1}{1+q} \sum_{n=0}^\infty \frac{(q;q^2)_n(-1)^n q^{n^2+n}}{(-q^2,-q^3;q^2)_n}\\[5pt]
&=\frac{1}{1+q} \lim_{a\to 0} 
\qhyp{3}{2}{q,q^2,q/a}{-q^2,-q^3}{q^2,aq}.
\end{align*}
Using the transformation formula for ${_3\phi_2}$ series \eqref{3phi2tr}
with  $(q,a,b,c,d,e)$ $\mapsto$  
$(q^2,q, q^2,q/a$, $-q^2,-q^3)$, we get
\begin{align*}
\sum_{n=0}^\infty \frac{(-1)^n q^{n^2+n}(q;q^2)_n}{(-q;q)_{2n+1}}
&=\frac{1}{1+q} \lim_{a\to 0} 
\frac{(q^2,q^2,aq^2;q^2)_\infty}{(-q^2,-q^3,aq;q^2)_\infty}\,
\qhyp{3}{2}{-q,-1,aq}{q^2,aq^2}{q^2,q^2}\\
&=\frac{1}{1+q} \frac{(q^2,q^2;q^2)_\infty}{(-q^2,-q^3;q^2)_\infty}
\sum_{n=0}^\infty \frac{(-1,-q;q^2)_n}{(q^2,q^2;q^2)_n} \, q^{2n}\\[5pt]
&=\frac{(q^2;q^2)_\infty^2}{(-q;q)_\infty} 
\sum_{n=0}^\infty \frac{(-1,-q;q^2)_n}{(q^2,q^2;q^2)_n} \, q^{2n}.
\end{align*}
Then applying \eqref{gfcd2} with 
$(q,a,b,c,d) \mapsto (q^2,q^2,0,-1,-q)$,  the proof is complete.
\end{proof}

In a similar manner, we can obtain several more false theta function 
identities of Ramanujan, such as Entry~9.3.3, Entry~9.4.2 and Entry~9.5.2 
as given in~\cite[Chapter 9]{AndBer05}:
\begin{align*}
&\sum_{n=0}^\infty \frac{(q;q^2)_n q^n}{(-q;q^2)_{n+1}}=\sum_{n=0}^\infty (-1)^n q^{2n(n+1)},\\[5pt]
&\sum_{n=0}^\infty \frac{(-1)^n q^{\binom{n+1}{2}}}{(-q;q)_{n}}=\sum_{n=0}^\infty q^{n(3n+1)/2} (1-q^{2n+1}),\\[5pt]
&\sum_{n=0}^\infty (q;q^2)_nq^n=\sum_{n=0}^\infty (-1)^n q^{3n^2+2n}(1+q^{2n+1}).
\end{align*}

Finally, from \eqref{gfcd2}, we can obtain 
most residual identities given by Lovejoy \cite{Lov12}. 
For example,  when $(a,b,c,d)\mapsto (aq,0,\sqrt{aq},-\sqrt{aq})$,  
\eqref{gfcd2} yields \cite[(2.11)]{Lov12}
\[
\sum_{n=0}^\infty \frac{(aq^{n+1};q)_nq^n}{(q;q)_n(aq^2;q^2)_n}=
\frac{1}{(q;q)_\infty(aq^2;q^2)_\infty} 
\sum_{n=0}^\infty (-1)^n a^nq^{n(n+1)},
\]
which was also given by Warnaar \cite[(5.3)]{War03}. 
The above identity is 
the residual identity corresponding to Ramanujan's partial theta function 
identity~\cite[Entry~6.3.11]{AndBer09}.

As a further specialization,
when $b,d\to 0 $ and $(a,c)\mapsto (a^2q,a)$, 
\eqref{gfcd2}
directly reduces to the following identity given by 
Lovejoy \cite[(2.5)]{Lov12}
\[
\sum_{n=0}^\infty \frac{(a;q)_nq^n}{(q,a^2q;q)_n}=
\frac{(aq;q)_\infty}{(q,a^2q;q)_\infty}
\sum_{n=0}^\infty (-1)^n a^{3n} q^{n(3n+1)/2}(1-a^2q^{2n+1}),
\]
which is the residual identity of Entry~6.3.6 in 
Ramanujan's Lost Notebook~\cite{AndBer09}.

We can also verify the following residual identity due to
Warnaar  \cite[p. 390]{War03}, see also Lovejoy \cite[(2.20)]{Lov12}.
\begin{corollary}
There holds
\[
\sum_{n=0}^\infty \frac{(-aq;q)_{n}q^{n}}{(q,a^2q;q)_n}
=\frac{(-aq;q)_\infty}{(q,a^2q;q)_\infty}
\Big(1-(1+a)\sum_{n=1}^\infty a^{3n-2}q^{n(3n-1)/2}(1-aq^n)\Big).
\]
\end{corollary}

\begin{proof}
By taking the limit $b,d\to 0$ and substituting 
$(a,c)\mapsto (a^2q,  -aq)$ 
in \eqref{gfcd2}, it follows that
\begin{align}
\sum_{n=0}^\infty \frac{(-aq;q)_nq^n}{(q,a^2q;q)_n}
&=\frac{(-aq;q)_\infty}{(q,a^2q;q)_\infty} 
\sum_{k=0}^\infty \frac{(-a;q)_k a^{3k} 
q^{\frac{3k(k+1)}{2}}(1-a^2q^{2k+1})}
{(-aq;q)_{k+1}}\nonumber\\[5pt]
&=\frac{(-a;q)_\infty}{(q,a^2q;q)_\infty}
\sum_{k=0}^\infty a^{3k}q^{\frac{3k(k+1)}{2}} \frac{1-a^2q^{2k+1}}{(1+aq^k)(1+aq^{k+1})}. \label{reswar}
\end{align}
By induction on $m$, it immediately follows that for all positive
integers $m$ we have
\begin{multline*}
(1+a) \sum_{k=0}^{m-1} a^{3k}q^{3k(k+1)/2} 
\frac{1-a^2q^{2k+1}}{(1+aq^k)(1+aq^{k+1})} \\
=1-a^{3m-2} q^{m(3m-1)/2} \frac{1+a}{1+aq^m}
-(1+a)\sum_{n=1}^{m-1} a^{3n-2}q^{n(3n-1)/2}(1-aq^n).
\end{multline*}
Letting $m$ tend to infinity and using the above to rewrite
the right hand side of \eqref{reswar}, completes the proof.
\end{proof}
 
Following the same steps as in
the above proof, when  $(a,b,c,d)$ $\mapsto$ $(aq$, $0$, $q\sqrt{a}$, $-q\sqrt{a})$, 
\eqref{gfcd2} leads to the following residual identity \cite[(2.32)]{Lov12}
\[
\sum_{n=0}^\infty \frac{(aq^2;q^2)_{n}q^{n}}{(q,aq;q)_n}
=\frac{1}{(q;q)_\infty(aq;q^2)_\infty(1+q)}\sum_{n=0}^\infty (-a)^nq^{n^2}(1+q^{2n+1}).
\]

Finally, by setting $(a,b,c,d)\mapsto (-q, 0, a, q/a)$ in \eqref{gfcd} and \eqref{gfcd2}, 
we obtain the following elegant result.

\begin{corollary}\label{cora-a} 
We have
\[
\sum_{n=0}^\infty \bigg(\frac{(a,q/a;q)_n}{(-a,-q/a;q)_{n+1}}+
\frac{(-a,-q/a;q)_n}{(a,q/a;q)_{n+1}}\bigg) 
q^{n^2+n}(1+q^{2n+1})=2\,\frac{(-q;q)_\infty^2(q^2;q^2)_\infty}
{(a^2,q^2/a^2;q^2)_\infty}.
\]
\end{corollary}

\section{The big $q$-Jacobi polynomials and partial theta functions}\label{section4}

In this section, we describe an interesting connection
between the big $q$-Jacobi polynomials and
partial theta functions. We also obtain a $q$-integral identity by 
considering the orthogonality of the big $q$-Jacobi polynomials.

The big $q$-Jacobi polynomials $P_n(x; a, b, c; q)$ were introduced by
Hahn \cite{Hahn49}, and  can be expressed in terms of 
basic hypergeometric functions as
\[
P_n(x;a,b,c;q)=\qhyp{3}{2}{q^{-n},abq^{n+1},x}{aq,cq}{q,q}.
\]
Ismail and Wilson \cite{IsmWil82} gave the following  generating
function for $P_n(x;a,b,c;q)$, 
\begin{equation}
\sum_{n=0}^\infty \frac{(cq;q)_n }{(bq,q;q)_n} P_{n}(x;a,b,c;q) t^n=
\qhyp{2}{1}{aqx^{-1},0}{aq}{q,x}\cdot 
\qhyp{1}{1}{bc^{-1}x}{bq}{q,cqt}.
\label{gf1}
\end{equation}

We begin by observing that 
the Andrews--Warnaar partial theta function identity \eqref{AndWar}
can be obtained from the generating function 
\eqref{gf1} by substituting 
$(a,b,c,x)\mapsto (a/q,b/q,b/q,q/t)$ and then
taking the limit as $t\to 0$.

In fact the summand considered in our main result \eqref{gfcd} 
can be expressed as a limiting
case of the big $q$-Jacobi polynomials.
To be more precise, substituting 
$(a,b,c,x)\mapsto (a/q, b/q, b/q,q/t)$ into the big $q$-Jacobi
polynomial $P_n(x;a,b,c;q)$, we find that 
\begin{align*}
\frac{(c,d;q)_n(abq^{n-1};q)_n q^n}{(q,a,b;q)_n}
&=\lim_{t\to 0} \frac{(c,d;q)_n t^n}{(q;q)_n}P_n(q/t;a/q,b/q,b/q;q) \\[5pt]
&= \lim_{t\to 0} \frac{(c,d;q)_n t^n}{(q;q)_n} 
\sum_{k=0}^n \frac{(q^{-n},abq^{n-1},q/t;q)_k}{(q,a,b;q)_k} \, q^k. 
\end{align*}

Andrews and Askey \cite{AndAsk85} found an explicit orthogonality
relation for the big $q$-Jacobi polynomials
\begin{multline}\label{ortho}
\int_{cq}^{aq}  \frac{(x/a,x/c;q)_\infty}{(x,xb/c;q)_\infty} \,
P_m(x;a,b,c;q)P_n(x;a,b,c;q) \, \textup{d}_q x\\
= aq(1-q)\,\frac{(q,c/a,aq/c,abq^2;q)_\infty}{(aq,bq,cq,abq/c;q)_\infty}\, 
\frac{(1-abq)}{(1-abq^{2n+1})}\,\frac{(q,bq,abq/c;q)_n}{(abq,aq,cq;q)_n}
\big({-}acq^2\big)^nq^{\binom{n}{2}} \delta_{mn}, 
\end{multline}
where the $q$-integral of a function $f(x)$, 
introduced by Jackson \cite{Jac10} and Thomae \cite{Tho69},
is defined by
\begin{equation}\label{qint}
\int_a^b f(x) \, \textup{d}_q x
=(1-q)\sum_{n=0}^\infty \big(bf(bq^n)-af(aq^n)\big)q^n.
\end{equation}

In \cite{Liu14}, Liu derived the following generating function for 
the big $q$-Jacobi polynomials
\[
\sum_{n=0}^\infty \frac{(1-abq^{2n+1})(abq,1/t;q)_nt^n}
{(q,q^2abt;q)_n}\,P_n(x;a,b,c;q)
=\frac{(abq,atq,ctq,x;q)_\infty}{(abtq^2,aq,cq,tx;q)_\infty}.
\]

By specialising  $c=b$ in the orthogonality relation \eqref{ortho},
then multiplying both sides by  
\[
\frac{(1-abq^{2n+1})(abq,1/t;q)_nt^n}{(q,q^2abt;q)_n}\,\frac{t^m}{(q;q)_m}
\]
and finally summing $m,n$ 
over the nonnegative integers, we obtain  
\begin{multline}\label{orthoi}
\int_{a}^b \frac{(xq/a,xq/b;q)_\infty}{(xqt,xq;q)_\infty}\,
\qhyp{2}{1}{b/x,0}{bq}{q,xqt} \cdot
\qhyp{1}{1}{xq}{aq}{q,aqt} \, \textup{d}_q x\\
=(b-a)(1-q)\,\frac{(q,aq/b,bq/a,abq^2t^2;q)_\infty}{(aq,bq,atq,btq;q)_\infty}.
\end{multline}
For $t=0$ both basic hypergeometric functions 
in the integrand trivialise to $1$. Then using \eqref{qint}, we obtain
\begin{equation}\label{abt}
a\,\frac{(aq/b;q)_\infty}{(aq;q)_\infty}\,\qhyp{2}{1}{aq,0}{aq/b}{q,q}
-b\,\frac{(bq/a;q)_\infty}{(bq;q)_\infty}\,\qhyp{2}{1}{bq,0}{bq/a}{q,q}
=a\,\frac{(b/a,aq/b;q)_\infty}{(aq,bq;q)_\infty}.
\end{equation}
By applying Heine's transformation 
formula \eqref{heine1} and then 
making the substitution $(a,b)\mapsto (-a/q,-b/q)$, 
we find that the above identity is equivalent to 
Ramanujan reciprocity theorem~\cite[Entry 6.3.3]{Ram88}
\begin{multline*}
\bigg(1+\frac{1}{b}\bigg)\sum_{n=0}^\infty \frac{(-1)^nq^{\binom{n+1}{2}}a^nb^{-n}}{(-aq;q)_n}-
\bigg(1+\frac{1}{a}\bigg)\sum_{n=0}^\infty \frac{(-1)^nq^{\binom{n+1}{2}}a^{-n}b^{n}}{(-bq;q)_n}\\
 =\bigg(\frac{1}{b}-\frac{1}{a}\bigg)\frac{(q,aq/b,bq/a;q)_\infty}{(-aq,-bq;q)_\infty}.
\end{multline*}

\end{document}